\documentclass[12pt]{article}
\usepackage{amsmath,amssymb}

\newtheorem{theorem}{Theorem}[section]
\newtheorem{proposition}[theorem]{Proposition}
\newtheorem{corollary}[theorem]{Corollary}

\newtheorem{remark}[theorem]{Remark}

\newenvironment{proof}{\smallskip\par{\sc Proof.}\enspace}%
 {{\unskip\nobreak\hfil\penalty50\hskip2em
          \hbox{}\nobreak\hfil{\rule[-1pt]{5pt}{10pt}}
          \parfillskip=0pt\finalhyphendemerits=0
          \par\medskip}} 

\makeatletter
\def\section{\@startsection {section}{1}{\z@}{3.25ex plus 1ex minus
 .2ex}{1.5ex plus .2ex}{\large\bf}}
\def\subsection{\@startsection{subsection}{2}{\z@}{3.25ex plus 1ex minus
 .2ex}{1.5ex plus .2ex}{\normalsize\bf}}
\@addtoreset{equation}{section} 
\makeatother

\begin{document}

\begin{center}
\LARGE\textsf{An extension of the Beckner's type Poincar\'e inequality to convolution measures on abstract Wiener spaces}
\end{center}

\vspace*{.3in}

\begin{center}
\sc
\large{Paolo Da Pelo\footnote{Dipartimento di Matematica, Universit\'a degli Studi di Bari Aldo Moro, Via E. Orabona 4, 70125 Bari - Italia. E-mail: \emph{paolo.dapelo@uniba.it}}\quad Alberto Lanconelli\footnote{Dipartimento di Matematica, Universit\'a degli Studi di Bari Aldo Moro, Via E. Orabona 4, 70125 Bari - Italia. E-mail: \emph{alberto.lanconelli@uniba.it}}
\quad Aurel I. Stan\footnote{Department of Mathematics, Ohio State University at Marion, 1465 Mount Vernon Avenue, Marion, OH 43302, U.S.A. E-mail: {\em stan.7@osu.edu}}}\\
\end{center}

\vspace*{.3in}

\begin{abstract}
We generalize the Beckner's type Poincar\'e inequality \cite{Beckner} to a large class of probability measures on an abstract Wiener space of the form $\mu\star\nu$, where $\mu$ is the reference Gaussian measure and $\nu$ is  a probability measure satisfying a certain integrability condition. As the Beckner inequality interpolates between the Poincar\'e and logarithmic Sobolev inequalities, we utilize a family of products for functions which interpolates between the usual point-wise multiplication and the Wick product. Our approach is based on the positivity of a quadratic form involving Wick powers and integration with respect to those convolution measures. Our dimension-independent results are compared with some very recent findings in the literature. In addition, we prove that in the finite dimensional case the class of densities of convolutions measures satisfies a point-wise covariance inequality.   
\end{abstract}



\allowdisplaybreaks

\section{Introduction}
In 1989 Beckner \cite{Beckner} proved the following inequality:
\begin{eqnarray}\label{beckner}
\int_{W}|f(w)|^2d\mu(w)-\int_{W}\big|e^{-\tau N}f(w)\big|^2d\mu(w)\leq (2-p)\int_{W}|Df(w)|^2d\mu(w)
\end{eqnarray}
where $1\leq p\leq 2$, $e^{-\tau}=\sqrt{p-1}$, $\mu$ is a standard Gaussian probability measure on the (possibly infinite dimensional) space $W$, $Df$ denotes a suitable gradient of $f$ and $N$ stands for the \emph{number} or \emph{Ornstein-Uhlenbeck} operator. Observe that when $p=1$ or equivalently $\tau=+\infty$ then (\ref{beckner}) coincides with the classic Poincar\'e inequality (\cite{Chernoff}, \cite{Nash}):
\begin{eqnarray*}
\int_{W}|f(w)|^2d\mu(w)-\Big(\int_{W}f(w)d\mu(w)\Big)^2\leq \int_{W}|Df(w)|^2d\mu(w).
\end{eqnarray*}
Moreover, utilizing the Nelson's hyper-contractive estimate \cite{Nelson}:
\begin{eqnarray*}
\int_{W}\big|e^{-\tau N}f(w)\big|^2d\mu(w)\leq\Big(\int_{W}|f(w)|^pd\mu(w)\Big)^{\frac{2}{p}}
\end{eqnarray*} 
one can rewrite (\ref{beckner}) as
\begin{eqnarray}\label{beckner 2}
\int_{W}|f(w)|^2d\mu(w)-\Big(\int_{W}|f(w)|^pd\mu(w)\Big)^{\frac{2}{p}}\leq (2-p)\int_{W}|Df(w)|^2d\mu(w).
\end{eqnarray}
Dividing both sides of (\ref{beckner 2}) by $2-p$ and letting $p\to 2^-$ one obtains the logarithmic Sobolev inequality (\cite{Gross}):
\begin{eqnarray*}
&&\int_{W}|f(w)|^2\ln\big(|f(w)|^2\big)d\mu(w)-\int_{W}|f(w)|^2d\mu(w)\cdot\ln\Big(\int_{W}|f(w)|^2d\mu(w)\Big)\nonumber\\
&&\leq 2\int_{W}|Df(w)|^2d\mu(w).
\end{eqnarray*}
Inequalities (\ref{beckner})-(\ref{beckner 2}), viewed as an interpolation between the Poincar\'e and logarithmic Sobolev inequalites, have attracted the attention of several authors; generalizations to log-concave measures, search for best constants and applications to partial differential equations have been the main topics of investigation. We refer the reader to the papers \cite{ABD}, \cite{AMTU}, \cite{LO}, \cite{Wang} and the references quoted there.\\

\noindent In the paper \cite{DLS 2} the authors introduced and studied a family of products for functions defined on Gaussian spaces:
\begin{eqnarray}\label{alpha intro}
f\circ_{\alpha} g:=e^{\tau N}\Big(e^{-\tau N}f\cdot e^{-\tau N}g\Big)
\end{eqnarray}
where $\sqrt{\alpha}:=e^{-\tau}$ and $f,g$ belong to some suitable function space (see Section 2.2 below for precise conditions). This family of products interpolates between the usual point-wise multiplication, when $\alpha=1$, and the Wick product, when $\alpha =0$ (this is obtained in \cite{DLS 2} through a limit argument). The crucial role of this family of products is in connection with the theory of stochastic integration and stochastic differential equations; in fact, one can prove that the following Wong-Zakai-type theorem holds:\\
\noindent  If for $k\geq1$, $\{W^k_t\}_{0\leq t\leq T}$ is a smooth approximation of the
white noise $W_t:=\frac{dB_t}{dt}$ ($B_t$ being a one dimensional Brownian motion) then the solution of
\begin{eqnarray*}
\frac{dX_t^k}{dt}=b(X_t^k)+X_t^k\circ_{\alpha} W_t^k,\quad
X_0^k=x
\end{eqnarray*}
converges in $\mathcal{L}^p(W,\mu)$, as $k$ goes to infinity, to the solution of
\begin{eqnarray*}
dX_t=b(X_t)dt+X_td^{\alpha}B_t,\quad X_0=x
\end{eqnarray*}
where
\begin{eqnarray*}
\int_0^TX_td^{\alpha}B_t:=\lim_{n\to+\infty}\sum_{k=1}^nX_{(1-\frac{\alpha}{2})t_{k-1}+\frac{\alpha}{2} t_{k}}\cdot(B_{t_{k}}-B_{t_{k-1}}).
\end{eqnarray*}
Observe that when $\alpha=0$ or $1$ we obtain the It\^o and Stratonovich integrals, respectively.\\

\noindent The aim of the present paper is to show that the Beckner's inequality (\ref{beckner}) can be generalized in a natural way to convolution measures on abstract Wiener spaces. This generalization passes through the use of the products $\circ_{\alpha}$ defined in (\ref{alpha intro}) and contains as a particular case the Poincar\'e-type inequality obtained in \cite{L}. More precisely, we will prove the following inequality:
\begin{eqnarray}\label{main result intro}
\int_{W}|f(w)|^2d\rho(w)-\int_{W} (f\circ_{\alpha}f)(w)d\rho(w)\leq (1-\alpha)\int_{W}\Vert Df(w)\Vert_H^2d\rho(w)
\end{eqnarray}
where $(W,H,\mu)$ is an abstract Wiener space, $\rho=\mu\star\nu$ and $\nu$ is a probability measure on $(W,\mathcal{B}(W))$. To see how (\ref{main result intro}) reduces to (\ref{beckner}) when $\rho=\mu$ observe that by definition $\alpha=p-1$ and that we can write
\begin{eqnarray*}
\int_{W} (f\circ_{\alpha}f)(w)d\mu(w)&=&\int_{W}e^{\tau N}\big(e^{-\tau N}f\cdot e^{-\tau N}f\big) (w)d\mu(w)\\
&=&\int_{W}e^{-\tau N}f(w)\cdot e^{-\tau N}f (w)d\mu(w)\\
&=&\int_{W}\big|e^{-\tau N}f(w)\big|^2 d\mu(w).
\end{eqnarray*}
Moreover, since $\circ_{\alpha}$ approaches the Wick product $\diamond$ as $\alpha\to 0^+$, inequality (\ref{main result intro}) becomes in that limit
\begin{eqnarray*}
\int_{W}|f(w)|^2d\rho(w)-\int_{W} (f\diamond f)(w)d\rho(w)\leq\int_{W}\Vert Df(w)\Vert_H^2d\rho(w).
\end{eqnarray*}
This last inequality, obtained in \cite{L}, is weaker than the classic Poincar\'e inequality for the measure $\rho$ since in general we have
\begin{eqnarray*}
\Big(\int_{W}f(w)d\rho(w)\Big)^2\leq\int_{W} (f\diamond f)(w)d\rho(w).
\end{eqnarray*}
Our approach is based on a novel idea whose crucial ingredient is the positive definiteness of a certain quadratic form involving Wick powers and integration with respect to convolution measures (see Proposition \ref{convolution measures} below). We mention that in the very recent papers \cite{Z}, \cite{WW} and \cite{CZ} Poincar\'e, weak Poincar\'e and logarithmic Sobolev inequalities for convolution measures on finite dimensional Euclidean spaces have been investigated: here the reference Gaussian (or log-concave) measure is convolved with compactly supported measures. We work in a dimension free framework and once we specify our assumptions for the finite dimensional case (see Corollary \ref{corollary} below) we get an exponential integrability condition on the measure $\nu$ (see (\ref{integrability condition}) below), which is clearly satisfied for compactly supported measures. However, as we mentioned above, inequality (\ref{main result intro}) is weaker, at least for $\alpha=0$, than the Poincar\'e inequality (and hence the logarithmic Sobolev inequality) studied in the above mentioned papers.\\
We first prove inequality (\ref{main result intro}) for $f$ being a linear combination of stochastic exponentials and then, under an additional condition on the integrating measure $\rho$, we extend the validity of the result by density to suitable Sobolev spaces, which clearly contain the class of smooth cylindrical functions (that usually represent the class for testing functional inequalities on infinite dimensional domains).\\
The paper is organized as follows: Section 2 collects definitions, notations and the necessary background material while in Section 3, after some preliminary results, we state and prove the main theorem of the paper followed by some refinements for the finite dimensional case and a point-wise covariance inequality (see (\ref{a}) below) satisfied by the densities of convolution measures with respect to the reference Gaussian measure. 

\section{Framework}

The aim of this section is to collect the necessary background
material and fix the notation. For the sake of clarity the topics
will not be treated in their greatest generality. For more details the interested reader is
referred to the books of Bogachev \cite{Bogachev}, Janson \cite{J}, Nualart \cite{Nualart} and
to the paper by Potthoff and Timpel \cite{PT} (the latter reference
is suggested, among other things, for the theory of the spaces
$\mathcal{G}_{\lambda}$ and the notion of Wick product).

\subsection{The spaces $\mathbb{D}^{k,p}$ and $\mathcal{G}_{\lambda}$}

Let $(H,W,\mu)$ be an \emph{abstract Wiener space}, that means
$(H,\langle\cdot,\cdot\rangle_H)$ is a separable Hilbert space which
is continuously and densely embedded in the Banach space
$(W,\Vert\cdot\Vert_W)$ and $\mu$ is a Gaussian probability measure
on the Borel sets $\mathcal{B}(W)$ of $W$ such that
\begin{eqnarray}\label{Gaussian characteristic}
\int_{W}e^{i\langle w,w^*\rangle}d\mu(w)=e^{-\frac{1}{2}\Vert
w^*\Vert_H^2},\quad\mbox{ for all }w^*\in W^*.
\end{eqnarray}
Here $W^*\subset H$ denotes the dual space of $W$ (which in turn is
dense in $H$) and $\langle\cdot,\cdot\rangle$ stands for the dual
pairing between $W$ and $W^*$. We will refer to $H$ as the
\emph{Cameron-Martin} space of $W$. Set for $p\geq 1$,
\begin{eqnarray*}
\mathcal{L}^p(W,\mu):=\Big\{f:W\to\mathbb{R}\mbox{ such that }\Vert
f\Vert_p:=\Big(\int_W|f(w)|^pd\mu(w)\Big)^{\frac{1}{p}}<+\infty\Big\}.
\end{eqnarray*}
It follows from
(\ref{Gaussian characteristic}) that the map
\begin{eqnarray*}
W^*&\to&\mathcal{L}^2(W,\mu)\\
w^*&\mapsto&\langle w,w^*\rangle
\end{eqnarray*}
is an isometry; we can therefore define for $\mu$-almost all $w\in
W$ the quantity $\langle w,h\rangle$ for $h\in H$  as an element of
$\mathcal{L}^2(W,\mu)$.\\
We now introduce the gradient operator and a class of functions
of Sobolev type. On the set
\begin{eqnarray*}
\mathcal{S}:=\{f(w)=\varphi(\langle w,h_1\rangle,...,\langle w,h_n\rangle)\mbox{ where
}n\in\mathbb{N}, h_1,...,h_n\in H\mbox{ and }\varphi\in
C_0^{\infty}(\mathbb{R}^n)\}
\end{eqnarray*}
define
\begin{eqnarray*}
D(\varphi(\langle w,h_1\rangle,...,\langle w,h_n\rangle)):=\sum_{j=1}^n\frac{\partial\varphi}{\partial
x_j}(\langle w,h_1\rangle,...,\langle w,h_n\rangle)h_j.
\end{eqnarray*}
The operator $D$ maps $\mathcal{S}$ into $\mathcal{L}^p(W,\mu;H)$;
moreover by means of the integration by parts formula
\begin{eqnarray*}
\int_W\langle Df(w),h\rangle_Hd\mu(w)=\int_Wf(w)\cdot\langle w,h\rangle d\mu(w),\quad f\in\mathcal{S}, h\in H
\end{eqnarray*}
one can prove that $D$ is closable in $\mathcal{L}^p(W,\mu)$; we
therefore define the space $\mathbb{D}^{1,p}$ to be closure of
$\mathcal{S}$ under the norm
\begin{eqnarray*}
\Vert f\Vert_{1,p}:=\Big(\int_W|f(w)|^pd\mu(w)+\int_W\Vert
Df(w)\Vert_H^pd\mu(w)\Big)^{\frac{1}{p}}.
\end{eqnarray*}
In a similar way, iterating the definition of $D$ and introducing
for any $k\in\mathbb{N}$ the norms
\begin{eqnarray*}
\Vert f\Vert_{k,p}:=\Big(\int_W|f(w)|^pd\mu(w)+\sum_{j=1}^k\int_W\Vert
D^jf(w)\Vert_{H^{\otimes j}}^pd\mu(w)\Big)^{\frac{1}{p}}.
\end{eqnarray*}
one constructs the spaces $\mathbb{D}^{k,p}$. \\
In order to prove our main results we need to introduce an
additional class of functions. To this aim recall that by the Wiener-It\^o chaos
decomposition theorem any element $f$ in $\mathcal{L}^2(W,\mu)$ has
an infinite orthogonal expansion
\begin{eqnarray*}
f=\sum_{n\geq 0} \delta^n(f_n),
\end{eqnarray*}
where $f_n\in H^{\hat{\otimes}n}$, the space of symmetric elements
of $H^{\otimes n}$, and $\delta^n(f_n)$ stands for the multiple
Wiener-It\^o integral of $f_n$. We remark that $\delta^1(f_1)$
coincides with the element $\langle w,f_1\rangle$ mentioned above. Moreover,
one has
\begin{eqnarray*}
\Vert f\Vert_2^2=\sum_{n\geq 0}n!\Vert f_n\Vert^2_{H^{\otimes n}}.
\end{eqnarray*}
It is useful to observe that if $f$ happens to be in
$\mathbb{D}^{1,2}$ then
\begin{eqnarray*}
\int_W\Vert Df(w)\Vert_H^2d\mu(w)=\sum_{n\geq 1}n n!\Vert f_n\Vert^2_{H^{\otimes
n}}.
\end{eqnarray*}
For any $\lambda\geq 0$ define the operator $\Gamma(\lambda)$ acting
on $\mathcal{L}^2(W,\mu)$ as
\begin{eqnarray*}
\Gamma(\lambda)\Big(\sum_{n\geq 0}\delta^n(f_n)\Big):=\sum_{n\geq 0}
\lambda^n\delta^n(f_n).
\end{eqnarray*}
Observe that with $\lambda=e^{-\tau}$, $\tau\geq 0$ then the operator $\Gamma(\lambda)$ coincides with the
Ornstein-Uhlenbeck semigroup
\begin{eqnarray*}
(P_{\tau}f)(w):=\int_Wf\big(e^{-\tau}w+\sqrt{1-e^{-2\tau}}\tilde{w}\big)d\mu(\tilde{w}),\quad
w\in W, \tau\geq 0
\end{eqnarray*}
which is a bounded  operator. Otherwise,
$\Gamma(\lambda)$ is an unbounded operator with domain in $\mathcal{L}^2(W,\mu)$ given by
\begin{eqnarray*}
\mathcal{G}_{\lambda}:=\Big\{f=\sum_{n\geq 0}
\delta^n(f_n)\in\mathcal{L}^2(W,\mu)\mbox{ such that }\Vert
f\Vert_{\mathcal{G}_{\lambda}}^2:=\sum_{n\geq 0}n!\lambda^{2n}\Vert
f_n\Vert^2_{H^{\otimes n}}<+\infty\Big\}.
\end{eqnarray*}
The family $\{\mathcal{G}_{\lambda}\}_{\lambda\geq 1}$ is a
collection of Hilbert spaces with the property that
\begin{eqnarray*}
\mathcal{G}_{\lambda_2}\subset\mathcal{G}_{\lambda_1}\subset\mathcal{L}^2(W,\mu)
\end{eqnarray*}
for $1<\lambda_1<\lambda_2$. Define
$\mathcal{G}:=\bigcap_{\lambda\geq 1}\mathcal{G}_{\lambda}$ endowed
with the projective limit topology; the space $\mathcal{G}$ turns
out to be a reflexive Fr\'echet space. Its dual $\mathcal{G}^*$ is a
space of generalized functions that can be represented as
$\mathcal{G}^*=\bigcup_{\lambda> 0}\mathcal{G}_{\lambda}$. We remark
that for $f\in\mathcal{L}^2(W,\mu)$ and $g\in\mathcal{G}$ one has
\begin{eqnarray*}
\langle\langle f,g\rangle\rangle=\int_Wf(w)g(w)d\mu(w)
\end{eqnarray*}
where $\langle\langle\cdot,\cdot\rangle\rangle$ stands for the dual
pairing between $\mathcal{G}^*$ and $\mathcal{G}$.\\
One of the most representative elements of $\mathcal{G}$ is the so
called \emph{stochastic exponential}
\begin{eqnarray*}
w\in W\mapsto\mathcal{E}(h)(w):=\exp\Big\{\langle w,h\rangle-\frac{\Vert
h\Vert_H^2}{2}\Big\},\quad h\in H.
\end{eqnarray*}
We recall that stochastic exponentials correspond among other
things to Radon-Nikodym derivatives, with respect to the underlying
Gaussian measure $\mu$, of probability measures on
$(W,\mathcal{B}(W))$ obtained through shifted copies of $\mu$ along
Cameron-Martin directions. Its membership to $\mathcal{G}$ can be
easily verified since the Wiener-It\^o chaos decomposition of
$\mathcal{E}(h)$ is obtained with $f_n=\frac{h^{\otimes n}}{n!}$.
Moreover the linear span of the stochastic exponentials, that we
denote with $\mathcal{E}$, is dense in $\mathcal{L}^p(W,\mu)$,
$\mathbb{D}^{k,p}$, for any $p\geq 1$ and $k\in\mathbb{N}$, and $\mathcal{G}$.

\subsection{The Wick and $\alpha$-products}

For $h,k\in H$ define
\begin{eqnarray*}
\mathcal{E}(h)\diamond\mathcal{E}(k):=\mathcal{E}(h+k).
\end{eqnarray*}
This is called the \emph{Wick product} of $\mathcal{E}(h)$ and
$\mathcal{E}(k)$. Extend this operation by linearity to
$\mathcal{E}$ to get a
commutative, associative and distributive (with respect to the sum)
multiplication. The Wick product is easily seen to
be an unbounded bilinear operator on the $\mathcal{L}^p(W,\mu)$ spaces; for instance, the Wick product
$f\diamond g$ of the two square integrable elements $f$ and $g$
lives in the distributional space $\mathcal{G}^*$.\\
Now, let $f,g\in\mathcal{L}^p(W,\mu)$ for some $p> 1$. For $\alpha\in ]0,1]$ define
\begin{eqnarray}\label{alpha}
(f\circ_{\alpha}g)(w):=\Gamma(1/\sqrt{\alpha})(\Gamma(\sqrt{\alpha})f\cdot\Gamma(\sqrt{\alpha})g)(w),\quad w\in W.
\end{eqnarray}
This is called the \emph{$\alpha$-product} of $f$ and $g$; it was introduced for the first time in \cite{DLS 2} in connection with stochastic integrals and stochastic differential equations. This family of products provides an interpolation between the usual point-wise multiplication (obtained trivially with $\alpha$=1) and the Wick product (obtained in the limit as $\alpha\to 0^+$). A simple calculation shows that
\begin{eqnarray}\label{exponential alpha}
\mathcal{E}(h_1)\circ_{\alpha}\mathcal{E}(h_2)=\mathcal{E}(h_1+h_2)e^{\alpha\langle h_1,h_2\rangle_H}.
\end{eqnarray}
The reader is referred to Theorem \ref{aurel} below for a sharp H\"older inequality for the family of products $\circ_{\alpha}$. 

\section{Main results}

In the sequel we will call \emph{convolution measure} on $(W,\mathcal{B}(W))$ any probability measure of the form $\mu\star\nu$ where $\mu$ is the reference Gaussian measure on $(W,\mathcal{B}(W))$, $\nu$ is a probability measure on $(W,\mathcal{B}(W))$ and 
\begin{eqnarray*}
(\mu\star\nu)(A):=\int_W\mu(A-w)d\nu(w),\quad A\in\mathcal{B}(W).
\end{eqnarray*}
We begin this section with a simple but crucial result: the description of the interplay between the Wick product and convolution measures.
\begin{proposition}\label{convolution measures}
Let $\nu$ be a probability measure on $(W,\mathcal{B}(W))$ and define $\rho:=\mu\ast\nu$. Then for every $z_1,...,z_n\in\mathbb{C}$ and $h_1,...,h_n\in H$ one has
\begin{eqnarray*}
\int_W\Big(\sum_{j=1}^nz_j\exp\Big\{i\langle w,h_j\rangle+\frac{\Vert h_j\Vert_H^2}{2}\Big\}\Big)\diamond\overline{\Big(\sum_{j=1}^nz_j\exp\Big\{i\langle w,h_j\rangle+\frac{\Vert h_j\Vert_H^2}{2}\Big\}\Big)}d\rho(w)\geq 0
\end{eqnarray*}  
where $i$ is the imaginary unit and $\overline{u}$ stands for the complex conjugate of $u$. 
\end{proposition}

\begin{proof}
We simply need to utilize the definition of Wick product and the Fourier transform characterization (\ref{Gaussian characteristic}) of the underlying Gaussian measure $\mu$:
\begin{eqnarray*}
&&\int_W\Big(\sum_{j=1}^nz_j\exp\Big\{i\langle w,h_j\rangle+\frac{\Vert h_j\Vert_H^2}{2}\Big\}\Big)\diamond\overline{\Big(\sum_{j=1}^nz_j\exp\Big\{i\langle w,h_j\rangle+\frac{\Vert h_j\Vert_H^2}{2}\Big\}\Big)}d\rho(w)\\
&=&\int_W\sum_{j,k=1}^nz_j\bar{z}_k\exp\Big\{i\langle w,h_j\rangle+\frac{\Vert h_j\Vert_H^2}{2}\Big\}\diamond\exp\Big\{-i\langle w,h_k\rangle+\frac{\Vert h_k\Vert_H^2}{2}\Big\}d\rho(w)\\
&=&\int_W\sum_{j,k=1}^nz_j\bar{z}_k\exp\Big\{i\langle w,h_j-h_k\rangle+\frac{\Vert h_j-h_k\Vert_H^2}{2}\Big\}d\rho(w)\\
&=&\sum_{j,k=1}^nz_j\bar{z}_k\exp\Big\{\frac{\Vert h_j-h_k\Vert_H^2}{2}\Big\}\int_W\exp\Big\{i\langle w,h_j-h_k\rangle\Big\}d\rho(w)\\
&=&\sum_{j,k=1}^nz_j\bar{z}_k\exp\Big\{\frac{\Vert h_j-h_k\Vert_H^2}{2}\Big\}\int_W\exp\Big\{i\langle w,h_j-h_k\rangle\Big\}d\mu(w)\int_W\exp\Big\{i\langle w,h_j-h_k\rangle\Big\}d\nu(w)\\
&=&\sum_{j,k=1}^nz_j\bar{z}_k\int_W\exp\Big\{i\langle w,h_j-h_k\rangle\Big\}d\nu(w)\\
&=&\int_W\Big(\sum_{j=1}^nz_j\exp\Big\{i\langle w,h_j\rangle\Big\}\Big)\cdot\overline{\Big(\sum_{j=1}^nz_j\exp\Big\{i\langle w,h_j\rangle\Big\}\Big)}d\nu(w)\\
&=&\int_W\Big|\sum_{j=1}^nz_j\exp\Big\{i\langle w,h_j\rangle\Big\}\Big|^2d\nu(w)\\
&\geq&0.
\end{eqnarray*}
\end{proof}

\begin{remark}\label{strong positivity}
Assume the measure $\rho$ from the previous proposition to be absolutely continuous with respect to $\mu$ with a density $\xi$ belonging to $\mathcal{L}^p(W,\mu)$ for some $p>1$. In this case we can write
\begin{eqnarray*}\label{positive definite}
&&\int_W\Big(\sum_{j=1}^nz_j\exp\Big\{i\langle w,h_j\rangle+\frac{\Vert h_j\Vert_H^2}{2}\Big\}\Big)\diamond\overline{\Big(\sum_{j=1}^nz_j\exp\Big\{i\langle w,h_j\rangle+\frac{\Vert h_j\Vert_H^2}{2}\Big\}\Big)}d\rho(w)\nonumber\\
&=&\int_W\Big(\sum_{j=1}^nz_j\exp\Big\{i\langle w,h_j\rangle+\frac{\Vert h_j\Vert_H^2}{2}\Big\}\Big)\diamond\overline{\Big(\sum_{j=1}^nz_j\exp\Big\{i\langle w,h_j\rangle+\frac{\Vert h_j\Vert_H^2}{2}\Big\}\Big)}\cdot\xi(w)d\mu(w)\nonumber\\
&=&\sum_{j,k=1}^nz_j\bar{z}_k\int_W\exp\Big\{i\langle w,h_j-h_k\rangle+\frac{\Vert h_j-h_k\Vert_H^2}{2}\Big\}\cdot\xi(w)d\mu(w)\nonumber\\
&=&\sum_{j,k=1}^nz_j\bar{z}_k\tau_{\xi}(h_j-h_k)
\end{eqnarray*}
where 
\begin{eqnarray*}
h\in H\mapsto \tau_{\xi}(h):=\int_W\exp\Big\{i\langle w,h\rangle+\frac{\Vert h\Vert_H^2}{2}\Big\}\cdot\xi(w)d\mu(w).
\end{eqnarray*}
With this notation the statement of Proposition \ref{convolution measures} reads
\begin{eqnarray*}
\sum_{j,k=1}^nz_j\bar{z}_k\tau_{\xi}(h_j-h_k)\geq 0
\end{eqnarray*}
which means that the function $\tau_{\xi}$ is positive definite; the latter is in turn equivalent, according to Proposition 5.1 in \cite{NZ}, to the property
\begin{eqnarray}\label{def. strong pos.}
\langle\langle \Gamma(1/\sqrt{\alpha})\xi,\varphi\rangle\rangle\geq 0\mbox{ for each non negative }\varphi\in\mathcal{G}\mbox{ and }\alpha>0.
\end{eqnarray} 
Here $\langle\langle\cdot,\cdot\rangle\rangle$ denotes the dual pairing between the distributional space $\mathcal{G}^*$ and the test function space $\mathcal{G}$. We mention that elements satisfying condition (\ref{def. strong pos.}) are referred in \cite{NZ} as \emph{strongly positive}.
\end{remark}

Another connection between convolution measures and Wick product is the following.

\begin{proposition}
Let $\rho_1:=\mu\star\nu_1$ and $\rho_2:=\mu\star\nu_2$ be convolution measures on $(W,\mathcal{B}(W))$ and assume the existence of $\xi_1,\xi_2\in\mathcal{L}^1(W,\mu)$ such that $d\rho_1=\xi_1d\mu$ and $d\rho_2=\xi_2d\mu$. Then for $\rho_3:=\mu\star\nu_1\star\nu_2$ one has $d\rho_3=\xi_1\diamond\xi_2 d\mu$.
\end{proposition}

\begin{proof}
Let $h\in H$; then
\begin{eqnarray*}
\int_W\exp\{i\langle w,h\rangle\}d\rho_3(w)&=&\int_W\exp\{i\langle w,h\rangle\}d(\mu\star\nu_1\star\nu_2)(w)\\
&=&\int_W\exp\{i\langle w,h\rangle\}d\mu(w)\cdot\int_W\exp\{i\langle w,h\rangle\}d\nu_1(w)\\
&&\times\int_W\exp\{i\langle w,h\rangle\}d\nu_2(w)\\
&=&\exp\Big\{-\frac{\Vert h\Vert_H^2}{2}\Big\}\cdot\int_W\exp\{i\langle w,h\rangle\}d\nu_1(w)\\
&&\times\int_W\exp\{i\langle w,h\rangle\}d\nu_2(w)\\
&=&\exp\Big\{\frac{\Vert h\Vert_H^2}{2}\Big\}\cdot\int_W\exp\{i\langle w,h\rangle\}d\rho_1(w)\\
&&\times\int_W\exp\{i\langle w,h\rangle\}d\rho_2(w)\\
&=&\exp\Big\{\frac{\Vert h\Vert_H^2}{2}\Big\}\cdot\int_W\exp\{i\langle w,h\rangle\}\xi_1(w)d\mu(w)\\
&&\times\int_W\exp\{i\langle w,h\rangle\}\xi_2(w)d\mu(w)\\
&=&\exp\Big\{-\frac{\Vert h\Vert_H^2}{2}\Big\}\cdot\int_W\exp\Big\{i\langle w,h\rangle+\frac{\Vert h\Vert_H^2}{2}\Big\}\xi_1(w)d\mu(w)\\
&&\times\int_W\exp\Big\{i\langle w,h\rangle+\frac{\Vert h\Vert_H^2}{2}\Big\}\xi_2(w)d\mu(w)\\
&=&\exp\Big\{-\frac{\Vert h\Vert_H^2}{2}\Big\}\cdot\int_W\exp\Big\{i\langle w,h\rangle+\frac{\Vert h\Vert_H^2}{2}\Big\}(\xi_1\diamond\xi_2)(w)d\mu(w)\\
&=&\int_W\exp\{i\langle w,h\rangle\}(\xi_1\diamond\xi_2)(w)d\mu(w)
\end{eqnarray*}
where we utilized the characterizing property of the Wick product
\begin{eqnarray*}
\int_W(f\diamond g)(w)\mathcal{E}(h)(w)d\mu(w)=\int_Wf(w)\mathcal{E}(h)(w)d\mu(w)\cdot\int_Wg(w)\mathcal{E}(h)(w)d\mu(w)
\end{eqnarray*}
which holds for any $h\in H$.
\end{proof}

The next theorem is a particular case of a more general result proved in \cite{Stan} where the reader is referred for the proof (the link between the theorem presented below and the results in the reference mentioned before is: $\Gamma(\lambda)(f\circ_{\alpha}g)=\Gamma(\lambda)f\circ_{\frac{\alpha}{\lambda^2}}\Gamma(\lambda)g$). It provides a H\"older inequality for the family of $\alpha$-products $\circ_{\alpha}$ which we will utilize to find the right function spaces for our extension of the Beckner's type Poincar\'e inequality.
\begin{theorem}\label{aurel}
Let $p,q,r>1$ and $\alpha\in [0,1]$ be such that 
\begin{eqnarray}\label{condition on parameters}
\frac{1}{r-\frac{1-\alpha}{1+\alpha}}=\frac{1+\alpha}{2(p-1)+2\alpha}+\frac{1+\alpha}{2(q-1)+2\alpha}.
\end{eqnarray}
Then for any $f\in\mathcal{L}^p(W,\mu)$ and $g\in\mathcal{L}^q(W,\mu)$ one has $\Gamma\big(\sqrt{(1+\alpha)/2}\big)(f\circ_{\alpha}g)\in\mathcal{L}^r(W,\mu)$. More precisely,
\begin{eqnarray}\label{holder}
\Big\Vert\Gamma\big(\sqrt{(1+\alpha)/2}\big)(f\circ_{\alpha}g)\Big\Vert_r\leq\Vert f\Vert_p\cdot\Vert g\Vert_q.
\end{eqnarray}
\end{theorem}

\begin{remark}
Observe that when $\alpha=1$ then $\circ_{\alpha}$ coincides with the usual point-wise product and (\ref{condition on parameters})-(\ref{holder}) become the classic H\"older inequality. On the other hand, when $\alpha=0$ then $\circ_{\alpha}$ coincides with the Wick product and  (\ref{condition on parameters})-(\ref{holder}) reduce to the H\"older-Young-Lieb inequality proved in \cite{DLS}. 
\end{remark}

We now make the first step towards the main result of the present paper. We are going to show that the left hand side of our main inequality (see (\ref{main inequality}) below) is non negative.

\begin{proposition}\label{left inequality theorem}
Let $\nu$ be a probability measure on $(W,\mathcal{B}(W))$ and choose $\alpha\in ]0,1]$. Assume that $\rho:=\mu\ast\nu$ is absolutely continuous with respect to $\mu$ with density $\xi$ belonging to $\mathcal{G}_{\sqrt{2/(1+\alpha)}}$. Then for any $f\in\mathcal{L}^{3+\alpha}(W,\mu)$ one gets
\begin{eqnarray}\label{left inequality}
\int_W|f(w)|^2d\rho(w)-\int_W (f\circ_{\alpha}f)(w)d\rho(w)\geq 0.
\end{eqnarray}
\end{proposition}

\begin{proof}
First of all observe that the integrals appearing in the left hand side of (\ref{left inequality}) are finite. In fact, by the Nelson hyper-contractive inequality we deduce that
\begin{eqnarray*}
\Vert\xi\Vert_{\frac{3+\alpha}{1+\alpha}}&=&\Big\Vert\Gamma\big(\sqrt{(1+\alpha)/2}\big)\Gamma\big(\sqrt{2/(1+\alpha)}\big)\xi\Big\Vert_{\frac{3+\alpha}{1+\alpha}}\\
&\leq&\Big\Vert\Gamma\big(\sqrt{2/(1+\alpha)}\big)\xi\Big\Vert_2\\
&<&+\infty
\end{eqnarray*}
which implies that $\xi\in\mathcal{L}^{\frac{3+\alpha}{1+\alpha}}(W,\mu)$. Therefore, using H\"older inequality we get
\begin{eqnarray*}
\int_W|f(w)|^2d\rho(w)&=&\int_W|f(w)|^2\cdot\xi(w) d\mu(w)\\
&\leq&\Big(\int_W|f(w)|^{3+\alpha}d\rho(w)\Big)^{\frac{2}{3+\alpha}}\cdot\Vert\xi\Vert_{\frac{3+\alpha}{1+\alpha}}\\
&=&\Vert f\Vert^2_{3+\alpha}\cdot\Vert\xi\Vert_{\frac{3+\alpha}{1+\alpha}}
\end{eqnarray*} 
where $\frac{3+\alpha}{2}$ is the conjugate exponent of $\frac{3+\alpha}{1+\alpha}$. This shows the finiteness of the first integral in (\ref{left inequality}). Concerning the second integral, note that for $\alpha\leq1$ one has $2(1+\alpha)\leq 3+\alpha$ which implies $\mathcal{L}^{3+\alpha}(W,\mu)\subset\mathcal{L}^{2(1+\alpha)}(W,\mu)$. Now choosing $p=q=2(1+\alpha)$ and $r=2$ in (\ref{condition on parameters}) we get from (\ref{holder}) that
\begin{eqnarray*}
\Big\Vert\Gamma\big(\sqrt{(1+\alpha)/2}\big)(f\circ_{\alpha}f)\Big\Vert_2\leq\Vert f\Vert^2_{2(1+\alpha)}
\end{eqnarray*}
which implies that $f\circ_{\alpha}f\in\mathcal{G}_{\sqrt{(1+\alpha)/2}}$ (under our assumption on $f$). Therefore, the integral 
\begin{eqnarray*}
\int_W (f\circ_{\alpha}f)(w)d\rho(w)=\int_W (f\circ_{\alpha}f)(w)\cdot\xi(w)d\mu(w) 
\end{eqnarray*}
is finite if $\xi\in\mathcal{G}_{\sqrt{2/(1+\alpha)}}$.\\
To prove inequality (\ref{left inequality}) we recall  (see Remark \ref{strong positivity} above) that the function $\xi$, being the density of a convolution measure, is strongly positive, i.e. $\Gamma(1/\sqrt{\alpha})\xi\geq 0$ (in distributional sense) for any $\alpha>0$. Hence using the definition of $f\circ_{\alpha}f$ we can write
\begin{eqnarray*}
\int_W (f\circ_{\alpha}f)(w)d\rho(w)&=&\int_W (f\circ_{\alpha}f)(w)\xi(w)d\mu(w)\\
&=&\int_W \Gamma(1/\sqrt{\alpha})(\Gamma(\sqrt{\alpha})f)^2(w)\cdot\xi(w)d\mu(w)\\
&=&\int_W(\Gamma(\sqrt{\alpha})f)^2(w)\cdot(\Gamma(1/\sqrt{\alpha})\xi)(w)d\mu(w)\\
&\leq&\int_W\Gamma(\sqrt{\alpha})f^2(w)\cdot(\Gamma(1/\sqrt{\alpha})\xi)(w)d\mu(w)\\
&=&\int_Wf^2(w)\cdot\xi(w)d\mu(w)\\
&=&\int_W|f(w)|^2d\rho(w)
\end{eqnarray*}
where in the inequality we utilized the Jensen inequality for the bounded operator $\Gamma(\sqrt{\alpha})$ and the convex function $x\mapsto x^2$. 
\end{proof}

We are now ready to prove the main theorem of the present paper.

\begin{theorem}\label{main theorem}
Let $\nu$ be a probability measure on $(W,\mathcal{B}(W))$ and choose $\alpha\in ]0,1]$. Assume that $\rho:=\mu\ast\nu$ is absolutely continuous with respect to $\mu$ with density $\xi$ belonging to $\mathcal{G}_{\sqrt{2/(1+\alpha)}}$. Then for every $f\in\mathbb{D}^{1,3+\alpha}$ one has
\begin{eqnarray}\label{main inequality}
\int_W|f(w)|^2d\rho(w)-\int_W(f\circ_{\alpha}f)(w)d\rho(w)\leq (1-\alpha)\int_W\Vert Df(w)\Vert_H^2d\rho(w)
\end{eqnarray}
or equivalently,
\begin{eqnarray}\label{main inequality 2}
\int_W|f(w)|^2d\rho(w)-\int_W|(\Gamma(\sqrt{\alpha})f)(w)|^2\cdot(\Gamma(1/\sqrt{\alpha})\xi)(w)d\mu(w)\nonumber\\
\leq  (1-\alpha)\int_W\Vert Df(w)\Vert_H^2d\rho(w).
\end{eqnarray}
\end{theorem}

\begin{remark}
Observe that for $\nu=\delta_0$, the Dirac measure concentrated at $0\in W$, the measure $\rho$ coincides with $\mu$ implying that $\xi\equiv 1$ and in particular $\Gamma(1/\sqrt{\alpha})\xi\equiv 1$. Inserting these quantities in (\ref{main inequality 2}) we recover the Beckner's type Poincar\'e inequality (\ref{beckner}).   
\end{remark}

\begin{proof}
For any $\alpha\in ]0,1]$ define the map
\begin{eqnarray}\label{T}
T_{\alpha}:\mathcal{E}&\to&\mathcal{E}\nonumber\\
f&\mapsto& T_{\alpha}(f):=f\circ_{\alpha} f-|f|^2+(1-\alpha)\Vert Df\Vert_H^2.
\end{eqnarray}
Since $f\in\mathcal{E}$ we can write
$f=\sum_{j=1}^n\lambda_j\mathcal{E}(h_j)$ for some
$\lambda_1,...,\lambda_n\in\mathbb{R}$ and $h_1,...,h_n\in H$. Now
substitute this expression into (\ref{T}) to obtain (recall identity (\ref{exponential alpha})),
\begin{eqnarray*}
T_{\alpha}(f)&=&\sum_{j,k=1}^n\lambda_j\lambda_k\mathcal{E}(h_j)\circ_{\alpha}\mathcal{E}(h_k)-
\sum_{j,k=1}^n\lambda_j\lambda_k\mathcal{E}(h_j)\cdot\mathcal{E}(h_k)\\
&&+(1-\alpha)\sum_{j,k=1}^n\lambda_j\lambda_k\mathcal{E}(h_j)\cdot\mathcal{E}(h_k)\langle
h_j,h_k\rangle_H\\
&=&\sum_{j,k=1}^n\lambda_j\lambda_k\mathcal{E}(h_j)\diamond\mathcal{E}(h_k)e^{\alpha\langle h_j,h_k\rangle_H}-
\sum_{j,k=1}^n\lambda_j\lambda_k\mathcal{E}(h_j)\diamond\mathcal{E}(h_k)e^{\langle h_j,h_k\rangle_H}\\
&&+(1-\alpha)\sum_{j,k=1}^n\lambda_j\lambda_k\mathcal{E}(h_j)\diamond\mathcal{E}(h_k)e^{\langle
h_j,h_k\rangle_H}\langle
h_j,h_k\rangle_H\\
&=&\sum_{j,k=1}^n\lambda_j\lambda_k\mathcal{E}(h_j)\diamond\mathcal{E}(h_k)\Big(e^{\alpha\langle h_j,h_k\rangle_H}-e^{\langle
h_j,h_k\rangle_H}+(1-\alpha)e^{\langle h_j,h_k\rangle_H}\langle
h_j,h_k\rangle_H\Big).
\end{eqnarray*}
We now integrate with respect to the measure $\rho$ the
first and last terms of the previous chain of equalities to obtain
\begin{eqnarray}\label{equality}
&&\int_WT_{\alpha}(f)(w)d\rho(w)\nonumber\\
&=&\sum_{j,k=1}^n\lambda_j\lambda_k\Big(e^{\alpha\langle h_j,h_k\rangle_H}-e^{\langle
h_j,h_k\rangle_H}+(1-\alpha)e^{\langle h_j,h_k\rangle_H}\langle
h_j,h_k\rangle_H\Big)\int_W(\mathcal{E}(h_j)\diamond\mathcal{E}(h_k))(w)d\rho(w)\nonumber\\
&=&\sum_{j,k=1}^n\lambda_j\lambda_ka_{jk}b_{jk},
\end{eqnarray}
where for $j,k\in\{1,...,n\}$ we set
\begin{eqnarray*}
a_{jk}:=e^{\alpha\langle h_j,h_k\rangle_H}-e^{\langle h_j,h_k\rangle_H}+(1-\alpha)e^{\langle
h_j,h_k\rangle_H}\langle h_j,h_k\rangle_H
\end{eqnarray*}
and
\begin{eqnarray*}
b_{jk}:=\int_W(\mathcal{E}(h_j)\diamond\mathcal{E}(h_k))(w)d\rho(w).
\end{eqnarray*}
Observe that the matrix $A=\{a_{jk}\}_{1\leq j,k\leq n}$ is positive
semi-definite; in fact, if in the Beckner's type Poincar\'e inequality
\begin{eqnarray*}
\int_W|f(w)|^2d\mu(w)-\int_W|(\Gamma(\sqrt{\alpha})f)(w)|^2d\mu(w)\leq  (1-\alpha)\int_W\Vert Df(w)\Vert_H^2d\mu(w)
\end{eqnarray*}
we take $f$ to be $\sum_{j=1}^n\lambda_j\mathcal{E}(h_j)$ one gets
\begin{eqnarray*}
\sum_{j,k=1}^n\lambda_j\lambda_k(e^{\langle h_j,h_k\rangle_H}-e^{\alpha\langle h_j,h_k\rangle_H})\leq
(1-\alpha)\sum_{j,k=1}^n\lambda_j\lambda_ke^{\langle h_j,h_k\rangle_H}\langle
h_j,h_k\rangle_H,
\end{eqnarray*}
which corresponds exactly to what we are claiming. On the other
hand, from Proposition \ref{convolution measures} the matrix $B=\{b_{jk}\}_{1\leq j,k\leq n}$ is positive
semi-definite . Therefore the matrix $A\Box B:=\{a_{jk}\cdot b_{jk}\}_{1\leq
j,k\leq n}$ (which corresponds to the Hadamard product of the matrix
$A$ with the matrix $B$) is also positive semi-definite (see for
instance Styan \cite{S}), that means
\begin{eqnarray*}
\sum_{j,k=1}^n\lambda_j\lambda_ka_{jk}b_{jk}\geq 0,
\end{eqnarray*}
for any $\lambda_1,...,\lambda_n\in\mathbb{R}$. From
(\ref{equality}) this corresponds to
\begin{eqnarray*}
\int_WT_{\alpha}(f)(w)d\rho(w)\geq 0\quad\mbox{ for all }f\in\mathcal{E}.
\end{eqnarray*}
Recalling the definition of $T_{\alpha}$ this is equivalent to
\begin{eqnarray*}
\int_W(f\circ_{\alpha}f)(w)d\rho(w)-\int_W|f(w)|^2d\rho(w)+(1-\alpha)\int_W\Vert Df(w)\Vert_H^2d\rho(w)\geq 0.
\end{eqnarray*}
which proves inequality (\ref{main
inequality}) for $f\in\mathcal{E}$.\\
The next step is to extend the validity of the last inequality to the whole $\mathbb{D}^{1,3+\alpha}$.\\
Since the measure $\rho$ is assumed to be absolutely continuous with respect to $\mu$ with density $\xi$ belonging to $\mathcal{G}_{\sqrt{2/(1+\alpha)}}\subset\mathcal{L}^{\frac{3+\alpha}{1+\alpha}}(W,\mu)$ we can control, via the H\"older inequality, the quantities
\begin{eqnarray*}
\int_W|f(w)|^2d\rho(w)\quad\mbox{ and }\quad\int_W\Vert Df(w)\Vert_H^2d\rho(w)
\end{eqnarray*}  
with
\begin{eqnarray*}
\Vert f\Vert_{3+\alpha}\quad\mbox{ and }\quad\Vert \Vert Df\Vert_H\Vert_{3+\alpha}
\end{eqnarray*}  
respectively, and exploit the density of the set $\mathcal{E}$ in $\mathbb{D}^{1,3+\alpha}$. Moreover,  Theorem \ref{aurel} guarantees that for any $\alpha\in [0,1]$ the bilinear map
\begin{eqnarray*}
(f,g)\mapsto f\circ_{\alpha}g
\end{eqnarray*}
is continuous from $\mathcal{L}^{3+\alpha}(W,\mu)\times\mathcal{L}^{3+\alpha}(W,\mu)$ into $\mathcal{G}_{\sqrt{(1+\alpha)/2}}$. This fact, together with the density of $\mathcal{E}$ in $\mathcal{L}^{3+\alpha}(W,\mu)$, completes the proof of (\ref{main inequality}).\\
Inequality (\ref{main inequality 2}) follows in the same manner through the self-adjointness of the operator $\Gamma(1/\sqrt{\alpha})$.
\end{proof}

\subsection{The finite dimensional case and a point-wise covariance inequality}

In the previous section we proved Theorem \ref{main theorem} under the assumptions that $\rho$ is a probability measure of convolution type, i.e. of the form $\rho=\mu\star\nu$, on a general abstract Wiener space with reference Gaussian measure $\mu$ and that $\rho$ is absolutely continuous with respect to $\mu$ with a density $\xi$ belonging to $\mathcal{G}_{\sqrt{2/(1+\alpha)}}$. \\
We now want to focus on finite dimensional abstract Wiener spaces and give easy-to-check sufficient conditions on $\nu$ which guarantee the existence of the above mentioned smooth density.\\
To this aim, consider the abstract Wiener space $W=H=\mathbb{R}^n$ with 
\begin{eqnarray}\label{finite dimensional gaussian}
\mu(A)=\int_A(2\pi)^{-\frac{n}{2}}\exp\Big\{-\frac{|w|^2}{2}\Big\}dw,\quad  A\in\mathcal{B}(\mathbb{R}^n)
\end{eqnarray}
where $|\cdot|$ denotes the Euclidean norm on $\mathbb{R}^n$. Let $\nu$ be a probability measure on $(\mathbb{R}^n,\mathcal{B}(\mathbb{R}^n))$ and define $\rho:=\mu\star\nu$. It is easy to see that the assumption of absolute continuity of $\rho$ with respect to $\mu$ is automatically verified in this finite dimensional framework and that
\begin{eqnarray*}
\xi(w)&:=&\frac{d\rho}{d\mu}(w)\\
&=&\int_{\mathbb{R}^n}\exp\Big\{\langle w,y\rangle-\frac{|y|^2}{2}\Big\}d\nu(y),\quad w\in\mathbb{R}^n.
\end{eqnarray*} 
Observe in addition that for each $y\in\mathbb{R}^n$ the function
\begin{eqnarray*}
w\mapsto \exp\Big\{\langle w,y\rangle-\frac{|y|^2}{2}\Big\}
\end{eqnarray*}
plays the role of stochastic exponential in the abstract Wiener space under consideration.  We have the following.
\begin{proposition}
Let $\nu$ be a probability measure on $(\mathbb{R}^n,\mathcal{B}(\mathbb{R}^n))$  and assume that 
\begin{eqnarray*}
\int_{\mathbb{R}^n}\exp\Big\{\frac{\lambda^2|y|^2}{2}\Big\}d\nu(y)<+\infty,\quad\mbox{ for some }\lambda>1.
\end{eqnarray*}
Then the probability measure $\rho:=\mu\star\nu$ is absolutely continuous with respect to $\mu$ with density belonging to $\mathcal{G}_{\lambda}$. 
\end{proposition}
\begin{proof}
We have only to check the membership of
\begin{eqnarray*}
\xi(w)=\int_{\mathbb{R}^n}\exp\Big\{\langle w,y\rangle-\frac{|y|^2}{2}\Big\}d\nu(y),\quad w\in\mathbb{R}^n
\end{eqnarray*}
to the space $\mathcal{G}_{\lambda}$. Using the Minkowski inequality we get that
\begin{eqnarray*}
\Vert\xi\Vert_{\mathcal{G}_{\lambda}}&=&\Vert\Gamma(\lambda)\xi\Vert_2\\
&=&\Big\Vert\Gamma(\lambda)\int_{\mathbb{R}^n}\exp\Big\{\langle \cdot,y\rangle-\frac{|y|^2}{2}\Big\}d\nu(y)\Big\Vert_2\\
&=&\Big\Vert\int_{\mathbb{R}^n}\exp\Big\{\langle \cdot,\lambda y\rangle-\frac{\lambda^2|y|^2}{2}\Big\}d\nu(y)\Big\Vert_2\\
&\leq&\int_{\mathbb{R}^n}\Big\Vert\exp\Big\{\langle \cdot,\lambda y\rangle-\frac{\lambda^2|y|^2}{2}\Big\}\Big\Vert_2 d\nu(y)\\
&=&\int_{\mathbb{R}^n}\exp\Big\{\frac{\lambda^2|y|^2}{2}\Big\}d\nu(y)\\
&<&+\infty.
\end{eqnarray*}
\end{proof}
We can therefore rephrase our main theorem with more transparent conditions.
\begin{corollary}\label{corollary}
Let $\alpha\in ]0,1]$ and consider a probability measure $\nu$ on $(\mathbb{R}^n,\mathcal{B}(\mathbb{R}^n))$  such that 
\begin{eqnarray}\label{integrability condition}
\int_{\mathbb{R}^n}\exp\Big\{\frac{|y|^2}{1+\alpha}\Big\}d\nu(y)<+\infty.
\end{eqnarray}
Define in addition $\rho:=\mu\star\nu$. Then for every $f\in\mathbb{D}^{1,3+\alpha}$ one has
\begin{eqnarray*}
\int_{\mathbb{R}^n}|f(w)|^2d\rho(w)-\int_{\mathbb{R}^n}(f\circ_{\alpha}f)(w)d\rho(w)\leq (1-\alpha)\int_{\mathbb{R}^n}\Vert Df(w)\Vert_H^2d\rho(w)
\end{eqnarray*}
or equivalently,
\begin{eqnarray*}
\int_{\mathbb{R}^n}|f(w)|^2d\rho(w)-\int_{\mathbb{R}^n}|(\Gamma(\sqrt{\alpha})f)(w)|^2\cdot(\Gamma(1/\sqrt{\alpha})\xi)(w)d\mu(w)\nonumber\\
\leq  (1-\alpha)\int_{\mathbb{R}^n}\Vert Df(w)\Vert_H^2d\rho(w)
\end{eqnarray*}
where
\begin{eqnarray*}
(\Gamma(1/\sqrt{\alpha})\xi)(w):=\int_{\mathbb{R}^n}\exp\Big\{\langle w,\frac{y}{\sqrt{\alpha}}\rangle-\frac{|y|^2}{2\alpha}\Big\}d\nu(y),\quad w\in\mathbb{R}^n.
\end{eqnarray*}
\end{corollary}

We conclude the paper with an additional result on convolution measures on $\mathbb{R}^n$. We know from before that,  if $\mu$ is the measure defined in (\ref{finite dimensional gaussian}) and $\nu$ is a probability measure on $(\mathbb{R}^n, \mathcal{B}(\mathbb{R}^n))$, then $\rho:=\mu\star\nu$ is absolutely continuous with respect to $\mu$ with density
\begin{eqnarray*}
\xi(w):=\int_{\mathbb{R}^n}\exp\Big\{\langle w,y\rangle-\frac{|y|^2}{2}\Big\}d\nu(y),\quad w\in\mathbb{R}^n.
\end{eqnarray*} 
We are going to show that functions of this type satisfy a point-wise covariance inequality, that means a point-wise inequality for functions which becomes after integration a covariance inequality in Gaussian spaces. 
\begin{proposition}
Let $\nu_1$ and $\nu_2$ be two probability measures on $(\mathbb{R}^n, \mathcal{B}(\mathbb{R}^n))$ and define for $w\in\mathbb{R}^n$,
\begin{eqnarray*}
\xi_1(w):=\int_{\mathbb{R}^n}\exp\Big\{\langle w,y\rangle-\frac{|y|^2}{2}\Big\}d\nu_1(y)\quad\mbox{ and }\quad\xi_2(w):=\int_{\mathbb{R}^n}\exp\Big\{\langle w,y\rangle-\frac{|y|^2}{2}\Big\}d\nu_2(y).
\end{eqnarray*}
Assume that for some $p>2$ the integrals
\begin{eqnarray}\label{1}
\int_{\mathbb{R}^n}\exp\Big\{(p-1)\frac{|y|^2}{2}\Big\}d\nu_i(y),\quad i=1,2
\end{eqnarray}
are finite. Then
\begin{eqnarray}\label{a}
\xi_1\cdot\xi_2\geq \xi_1\diamond\xi_2+\sum_{k=1}^n\partial_{x_k}\xi_1\diamond\partial_{x_k}\xi_2\quad\mbox{ in}\quad\mathcal{G}^*
\end{eqnarray}
i.e., for any non negative $\varphi\in\mathcal{G}$ one has
\begin{eqnarray}\label{b}
\langle\langle\xi_1\cdot\xi_2,\varphi\rangle\rangle-\langle\langle\xi_1\diamond\xi_2,\varphi\rangle\rangle-\sum_{k=1}^n\langle\langle\partial_{x_k}\xi_1\diamond\partial_{x_k}\xi_2,\varphi\rangle\rangle\geq 0.
\end{eqnarray}
\end{proposition}

\begin{proof}
First of all observe that condition (\ref{1}), due to Minkoski integral inequality, guarantees that $\xi_1$ and $\xi_2$ belong to $\mathcal{L}^p(\mathbb{R}^n,\mu)$ with $p>2$ and hence that all the terms appearing in  (\ref{a}) live in the distributional space $\mathcal{G}^*$. Now, for a  non negative $\varphi\in\mathcal{G}$ we can write
\begin{eqnarray*}
\langle\langle\xi_1\cdot\xi_2,\varphi\rangle\rangle&=&\langle\langle \int_{\mathbb{R}^n}\exp\Big\{\langle \cdot,y\rangle-\frac{|y|^2}{2}\Big\}d\nu_1(y)\cdot\int_{\mathbb{R}^n}\exp\Big\{\langle \cdot,y\rangle-\frac{|y|^2}{2}\Big\}d\nu_2(y),\varphi\rangle\rangle\\
&=&\langle\langle \int_{\mathbb{R}^n}\exp\Big\{\langle \cdot,y+z\rangle-\frac{|y+z|^2}{2}\Big\}\exp{\langle y,z\rangle}d\nu_1(y)d\nu_2(z),\varphi\rangle\rangle\\
&=&\int_{\mathbb{R}^n}\langle\langle\exp\Big\{\langle \cdot,y+z\rangle-\frac{|y+z|^2}{2}\Big\},\varphi\rangle\rangle\exp{\langle y,z\rangle}d\nu_1(y)d\nu_2(z)\\
&\geq&\int_{\mathbb{R}^n}\langle\langle\exp\Big\{\langle \cdot,y+z\rangle-\frac{|y+z|^2}{2}\Big\},\varphi\rangle\rangle(1+\langle y,z\rangle)d\nu_1(y)d\nu_2(z)\\
&=&\int_{\mathbb{R}^n}\langle\langle\exp\Big\{\langle \cdot,y+z\rangle-\frac{|y+z|^2}{2}\Big\},\varphi\rangle\rangle d\nu_1(y)d\nu_2(z)\\
&&+\int_{\mathbb{R}^n}\langle\langle\exp\Big\{\langle \cdot,y+z\rangle-\frac{|y+z|^2}{2}\Big\},\varphi\rangle\rangle\langle y,z\rangle d\nu_1(y)d\nu_2(z)\\
&=&\langle\langle \int_{\mathbb{R}^n}\exp\Big\{\langle \cdot,y\rangle-\frac{|y|^2}{2}\Big\}d\nu_1(y)\diamond\int_{\mathbb{R}^n}\exp\Big\{\langle \cdot,y\rangle-\frac{|y|^2}{2}\Big\}d\nu_2(y),\varphi\rangle\rangle\\
&&\sum_{k=1}^n\langle\langle\partial_{x_k}\int_{\mathbb{R}^n}\exp\Big\{\langle \cdot,y\rangle-\frac{|y|^2}{2}\Big\}d\nu_1(y)\diamond\partial_{x_k}\int_{\mathbb{R}^n}\exp\Big\{\langle \cdot,y\rangle-\frac{|y|^2}{2}\Big\}d\nu_2(y),\varphi\rangle\rangle\\
&=&\langle\langle\xi_1\diamond\xi_2,\varphi\rangle\rangle-\sum_{k=1}^n\langle\langle\partial_{x_k}\xi_1\diamond\partial_{x_k}\xi_2,\varphi\rangle\rangle.
\end{eqnarray*}
The proof is complete.
\end{proof}

\begin{remark}
If in (\ref{b}) we take $\varphi\equiv 1$ and assume $\xi_1$ and $\xi_2$ belonging to $\mathbb{D}^{1,p}$ then we obtain
\begin{eqnarray*}
\int_{\mathbb{R}^n}\xi_1(w)\cdot\xi_2(w)d\mu(w)&\geq&\int_{\mathbb{R}^n}\xi_1(w)d\mu(w)\cdot\int_{\mathbb{R}^n}\xi_2(w)d\mu(w)\\
&&+\sum_{k=1}^n\int_{\mathbb{R}^n}\partial_{x_k}\xi_1(w)d\mu(w)\cdot\int_{\mathbb{R}^n}\partial_{x_k}\xi_2(w)d\mu(w)
\end{eqnarray*}
which is the finite dimensional version of the covariance inequality obtained in \cite{Hu} for convex functions. Here we utilized the properties
\begin{eqnarray*}
\langle\langle f,1\rangle\rangle=\int_{\mathbb{R}^n}f(w)d\mu(w)\quad\mbox{ when }\quad f\in\mathcal{L}^{2}(\mathbb{R}^n,\mu)
\end{eqnarray*}
and
\begin{eqnarray*}
\int_{\mathbb{R}^n}(f\diamond g)(w)d\mu(w)=\int_{\mathbb{R}^n}f(w)d\mu(w)\cdot\int_{\mathbb{R}^n}g(w)d\mu(w).
\end{eqnarray*}
\end{remark}


\vspace*{8pt}

\begin{thebibliography}{99}

\bibitem{ABD}
A. Arnold, J.-P. Bartier and J. Dolbeault, Interpolation between logarithmic Sobolev and Poincar\'e inequalities, \emph{Commun. Math. Sci.} {\bf 5} (2007) 971-979.

\bibitem{AMTU}
A. Arnold, P. Markowich, G. Toscani and A. Unterreiter, On logarithmic Sobolev inequalities and the rate of convergence to equilibrium for Fokker-Planck type equations,
\emph{Comm. Partial Differential Equations} {\bf 26} (2001) 43-100.

\bibitem{Beckner} 
W. Beckner, A generalized Poincar\'e inequality for Gaussian measures, \emph{Proc. Amer. Math.
Soc.} {\bf 105} (1989) 397-400.

\bibitem{Bogachev}
V.I. Bogachev, \emph{Gaussian Measures}, American Mathematical Society, Providence, 1998.

\bibitem{CZ}
L.-J. Cheng and S.-Q. Zhang, Weak Poincar\'e Inequality for Convolution
Probability Measures, \emph{arXiv:1407.4910v1}.

\bibitem{Chernoff}
H. Chernoff, A note on an inequality involving the normal
distribution, \emph{Ann. Probab} \textbf{9} (1981) 533-535.

\bibitem{DLS}
P. Da Pelo, A. Lanconelli and A. I. Stan, A H\"older-Young-Lieb
inequality for norms of Gaussian Wick products, \emph{Inf. Dim.
Anal. Quantum Prob. Related Topics} \textbf{14} (2011) 375-407.

\bibitem{DLS 2}
P. Da Pelo, A. Lanconelli and A. I. Stan, An It\^o formula for a
family of stochastic integrals and related Wong-Zakai theorems,
\emph{Stochastic Processes and their Applications} \textbf{123}
(2013) 3183-3200.

\bibitem{Stan}
P. Da Pelo, A. Lanconelli and A.I. Stan, A H\"older-Young inequality for norms of generalized Gaussian Wick products, \emph{Preprint} (2014).

\bibitem{Gross}
L. Gross, Logarithmic Sobolev inequality, \emph{Amer. J. Math.}
\textbf{97} (1975) 1061-1083.

\bibitem{Hu}
Y. Hu,  It\^o-Wiener chaos expansion with exact residual and correlation, variance
inequalities, \emph{J. Theor. Probab.} {\bf 10} (1997) 835-848.

\bibitem{J}
S. Janson, \emph{Gaussian Hilbert spaces}, Cambridge Tracts in
Mathematics, 129. Cambridge University Press, Cambridge, 1997.

\bibitem{L}
A. Lanconelli, A new approach to Poincar\'e-type inequalities on the
Wiener space, \emph{Preprint} (2014).

\bibitem{LO}
R. Latala and K. Oleszkiewicz, Between Sobolev and Poincar\'e, \emph{Lecture Notes in Math.}
{\bf 1745} (2000) 147–168.

\bibitem{Nash}
J. Nash, Continuity of solutions of partial and elliptic equations,
\emph{Amer. J. Math.} \textbf{80} 931-954.

\bibitem{Nelson}
E. Nelson, The free Markoff field, \emph{J . Funct. Anal.} {\bf 12} (1973) 211-227.

\bibitem{Nualart}
D. Nualart, \emph{Malliavin calculus and Related Topics, II
edition}, Springer, New York 2006.

\bibitem{NZ}
D. Nualart and M. Zakai, Positive and strongly positive Wiener
functionals, {\em Barcelona Seminar on Stochastic Analysis Progr.
Probab.} {\bf 32} (1991) 132--146.

\bibitem{PT}
J. Potthoff and M. Timpel, On a dual pair of spaces of smooth and
generalized random variables, {\em Potential Analysis} {\bf 4}
(1995) 637--654.

\bibitem{S}
G. Styan, Hadamard product and multivariate statistical analysis,
\emph{Linear algebra and its applications} \textbf{6} (1973)
217-240.

\bibitem{Wang}
F. Y. Wang, A generalization of Poincar\'e and log-Sobolev inequalities, \emph{Potential Analysis}
{\bf 22} (2005) 1-15.

\bibitem{WW}
F. Y. Wang and J. Wang, Functional inequalities for convolution probability measures, \emph{arXiv:1308.1713}.

\bibitem{Z}
D. Zimmermann, Logarithmic Sobolev inequalities for mollified compactly supported measures, \emph{J. Funct. Anal.} {\bf 265} (2013)
1064-1083.


\end{thebibliography}
\end{document}